\documentclass[reqno,11pt]{amsart}

\usepackage[latin1]{inputenc}
\usepackage[english]{babel}

\usepackage{amscd}

\usepackage{amsmath}
\usepackage{amsthm}
\usepackage{amssymb}

\usepackage[notref, notcite, final]{showkeys}
\usepackage{color}
\usepackage{graphicx}
\usepackage[hidelinks]{hyperref}
\usepackage{cleveref}
\usepackage{fullpage}
\usepackage{comment}
\usepackage[shortlabels]{enumitem}

\usepackage{dsfont}

\newcommand{\pp}{\partial}

\newcommand{\rr}{\mathbb{R}}
\newcommand{\cc}{\mathbb{C}}

\renewcommand{\Re}{\mathrm{Re}}

\usepackage[autostyle]{csquotes}

\newtheorem{theorem}{Theorem}[section]
\newtheorem{lemma}[theorem]{Lemma}
\newtheorem{proposition}[theorem]{Proposition}
\newtheorem{corollary}[theorem]{Corollary}

\newtheorem{problem}[theorem]{Problem}

\theoremstyle{definition}
\newtheorem{definition}[theorem]{Definition}

\theoremstyle{remark}
\newtheorem{remark}[theorem]{Remark}
\crefname{enumi}{}{}

\title{\bf The Poisson kernel and
 the Fourier transform \\ of the  slice monogenic Cauchy kernels}

\author[F. Colombo]{Fabrizio Colombo}
\address{(FC)
Politecnico di Milano\\Dipartimento di Matematica\\Via E. Bonardi, 9\\20133
Milano, Italy}
\email{fabrizio.colombo@polimi.it}

\author[A. De Martino]{Antonino De Martino}
\address{(AD)
Politecnico di Milano\\Dipartimento di Matematica\\Via E. Bonardi, 9
\\20133
Milano, Italy}
\email{antonino.demartino@polimi.it}

\author[T. Qian]{Tao Qian}
\address{(TQ)
Depaerment of Mathematics\\
Faculty of Science and Technology\\
University of Macau, Macau\\
} \email{fsttq@umac.mo}

\author[I. Sabadini]{Irene Sabadini}
\address{(IS)
 Politecnico di Milano\\Dipartimento di Matematica\\Via E. Bonardi, 9\\20133
Milano, Italy
} \email{irene.sabadini@polimi.it}

\begin{document}
\maketitle

\begin{abstract}
The Fueter-Sce-Qian (FSQ for short)  mapping theorem is a two-steps procedure
to extend holomorphic functions of one complex variable
to slice monogenic functions and to monogenic functions.
Using the Cauchy formula of slice monogenic functions the
FSQ-theorem admits an integral representation for $n$ odd.
In this paper we show that the relation
$
\Delta_{n+1}^{(n-1)/2}S_L^{-1}=\mathcal{F}^L_n
$
 between the slice monogenic Cauchy kernel $S_L^{-1}$ and the F-kernel $\mathcal{F}^L_n$, that appear in the integral form of the FSQ-theorem for $n$ odd,
holds also in the case we consider the fractional powers of the Laplace operator $\Delta_{n+1}$ in dimension $n+1$, i.e., for $n$ even. Moreover, this
relation is proven computing explicitly Fourier transform of the kernels $S_L^{-1}$ and $\mathcal{F}^L_n$ as functions of the Poisson kernel.
Similar results hold for the right kernels
$S_R^{-1}$ and of $\mathcal{F}^R_n$.
\end{abstract}

\medskip
\noindent AMS Classification .


\noindent {\em }
\date{today}
\tableofcontents

\section{Introduction}

The Fueter-Sce-Qian (FSQ for short) mapping theorem is one of the deepest results in complex and  hypercomplex analysis because it shows how to extend holomorphic functions
of one complex variable to high dimensions for vector-valued functions.
This theorem is due to R. Fueter \cite{fueter1} for the quaternionic setting, it was generalized by
M. Sce \cite{sce}, to Clifford algebra $\mathbb{R}_n $ for $n$ odd while
the case of even dimension was proved by T. Qian in \cite{qian} (see also the recent monograph \cite{booktao}). The method of T. Qian requires the use of the Fourier transform in the space of distributions and is deeply different from the method of R. Fueter and M. Sce.
In the literature it is less known that the results of M. Sce, in \cite{sce},
 are written in a setting that
contains, as particular case, the Clifford algebra $\mathbb{R}_n $
of odd dimension, but, for example it also works for the octonions.
 For more details see the recent translation of the work of M. Sce with commentaries
 \cite[chapter 5]{SCEBOOK}.

\medskip
Consider  holomorphic functions of one complex variable
$f:\Omega \subseteq \mathbb{C} \to \mathbb{C}$ (denoted the set by $\mathcal{O}(\Omega)$).
It is well known that the way to extend the function theory of one complex variable to several complex variables is to
 consider the systems of Cauchy-Riemann equations,
 for
 $$
 f:\Pi\subseteq \mathbb{C}^n \to \mathbb{C},
 $$
 where $\Pi$ is an open set.
To explain the FSQ mapping theorem we need some preliminary notation
for the Clifford algebra setting.

\medskip
Let $\rr_n$ be the real Clifford algebra over $n$ imaginary units $e_1,\ldots ,e_n$
satisfying the relations $e_\ell e_m+e_me_\ell=0$,\  $\ell\not= m$, $e_\ell^2=-1.$
 An element in the Clifford algebra will be denoted by $\sum_A e_Ax_A$ where
$A=\{ \ell_1\ldots \ell_r\}\in \mathcal{P}\{1,2,\ldots, n\},\ \  \ell_1<\ldots <\ell_r$
 is a multi-index
and $e_A=e_{\ell_1} e_{\ell_2}\ldots e_{\ell_r}$, $e_\emptyset =1$.
An element $(x_0,x_1,\ldots,x_n)\in \mathbb{R}^{n+1}$  will be identified with the element
$
 x=x_0+\underline{x}=x_0+ \sum_{\ell=1}^nx_\ell e_\ell\in\mathbb{R}_n
$
called paravector and the real part $x_0$ of $x$ will also be denoted by $\Re(x)$.
The norm of $x\in\mathbb{R}^{n+1}$ is defined as $|x|^2=x_0^2+x_1^2+\ldots +x_n^2$.
 The conjugate of $x$ is defined by
$
\bar x=x_0-\underline x=x_0- \sum_{\ell=1}^nx_\ell e_\ell.
$
We denote by $\mathbb{S}$ the sphere
$$
\mathbb{S}=\{ \underline{x}=e_1x_1+\ldots +e_nx_n\ | \  x_1^2+\ldots +x_n^2=1\};
$$
for $I\in\mathbb{S}$ we obviously have $I^2=-1$.
Given an element $x=x_0+\underline{x}\in\rr^{n+1}$ let us set
$
I_x=\underline{x}/|\underline{x}|$ if $\underline{x}\not=0,
$
 and given an element $x\in\rr^{n+1}$, the set
$$
[x]:=\{y\in\rr^{n+1}\ :\ y=x_0+I |\underline{x}|, \ I\in \mathbb{S}\}
$$
is an $(n-1)$-dimensional sphere in $\mathbb{R}^{n+1}$.
The vector space $\mathbb{R}+I\mathbb{R}$ passing through $1$ and
$I\in \mathbb{S}$ will be denoted by $\mathbb{C}_I$ and
an element belonging to $\mathbb{C}_I$ will be indicated by $u+Iv$, for $u$, $v\in \mathbb{R}$.
With an abuse of notation we will write $x\in\mathbb{R}^{n+1}$.
Thus, if $U\subseteq\mathbb{R}^{n+1}$ is an open set,
a function $f:\ U\subseteq \mathbb{R}^{n+1}\to\mathbb{R}_n$ can be interpreted as
a function of the paravector $x$.

\medskip
We can now summarize the FSQ mapping theorem highlighting the two steps of the extension procedure as follows.
In the first step we obtain
slice monogenic functions \cite{GLOBAL,SOREN1,cstrends,Entirebook,CSSf,CSSd,CSSe,Cnudde,REN4},
see the book \cite{NONCOMMBOOK},
while in the second step we get the classical monogenic functions \cite{bds,BLU,DELSOSOU,gimu,Gurlebeck:2008}.
Let $\tilde{f}(z) = f_0(u,v)+i f_1(u,v)$, where $i=\sqrt{-1}$, be a  holomorphic function
 defined in a symmetric domain with respect to the real axis $D\subseteq\mathbb C$  and let
$$
\Omega _D= \{x =x_0+\underline{x}\ \ :\ \  (x_0, |\underline{x}|) \in D\}
$$
be the open set, induced by $D$, in $\mathbb{R}^{n+1}$.
Moreover, we assume that
 $$
f_0(u,-v)= f_0(u,v) \ \ \ {\rm and} \ \ \  f_1(u,-v)= -f_1(u,v)
$$
 namely $f_0$ and $f_1$ are, respectively, even and odd functions in the variable $v$.
 Additionally the pair $(f_0,f_1)$ satisfies the Cauchy-Riemann system. Then the FSQ extension procedure is as follows.

Step (I). The linear operator $T_{FSQ1}$ defined as
 $$
T_{FSQ1}(\tilde{f}(z)):=
f_0(x_0,|\underline{x}|)+\frac{\underline{x}}{|\underline{x}|}f_1(x_0,|\underline{x}|)\ \ \ {\rm on} \ \ \Omega_D
$$
extends the holomorphic function $\tilde{f}(z)$ to the slice monogenic function
$$
f(x):=f_0(x_0,|\underline{x}|)+\frac{\underline{x}}{|\underline{x}|}f_1(x_0,|\underline{x}|).
$$

Step (II). Consider the linear operator $T_{FSQ2}:=\Delta_{n+1}^{\frac{n-1}{2}}$ where $\Delta_{n+1}$ is the Laplace operator in $n+1$ dimensions, i.e.,
$\Delta_{n+1}=\pp^2_{x_0}+\sum_{j=1}^n\pp^2_{x_j}$.
Then, $T_{FSQ2}$
 maps the slice monogenic function $f(x)$ in the monogenic function
 $$
\breve{f}(x):=\textcolor{black}{T_{FSQ2}} \Big(\textcolor{black}{f_0(x_0,|\underline{x}|)
+\frac{\underline{x}}{|\underline{x}|}f_1(x_0,|\underline{x}|)}\Big),
$$
i.e., $\breve{f}(x)$
is in the kernel of the Dirac operator $D$, i.e.
$$
D\breve{f}(x):=\pp_{x_0}\breve{f}(x)+\sum_{i=1}^n e_i\pp_{x_i}\breve{f}(x)=0
\ \ \ {\rm on} \ \ \Omega_D.
$$
We point out that in the extension procedure the operator $T_{FSQ1}$ maps holomorphic functions
into the set of intrinsic slice monogenic functions, denoted by $\mathcal{N}(\Omega_D)$, that is strictly
contained in the set of slice monogenic functions $\mathcal{SM}(U)$.
Similarly, when we apply the operator $T_{FSQ2}$ to the set of slice monogenic functions, not necessarily intrinsic, we obtain a subclass of the monogenic functions that are called
 axially monogenic functions and are denoted by
 $\mathcal{AM}(\Omega_D)$, so we can visualize the FSQ construction  by the diagram:
\begin{equation*}
\begin{CD}
\textcolor{black}{\mathcal{O}(D)}  @>T_{FSQ1}>> \textcolor{black}{\mathcal{N}(\Omega_D)}  @>\ \   T_{FSQ2}=\Delta_{n+1}^{(n-1)/2} >>\textcolor{black}{\mathcal{AM}(\Omega_D)},
\end{CD}
\end{equation*}
where $T_{FSQ1}$ denotes the first linear operator and $T_{FSQ2}$ the second one.
As it is clear $\Delta_{n+1}^{\frac{n-1}{2}}$  is a fractional operator for $n$ even.

\medskip
The FSQ-mapping theorem and its generalizations can be found in \cite{E1,E2,E3,kqs,P1,P2,sommen1},
more recently there has been an intensive research in the direction of the
inverse FSQ-mapping theorem, which has been investigated in the papers \cite{A,CoSaSo1,CoSaSo2,C,CCC,D}.
Using the Radon and dual Radon transform, see \cite{radon},
it is possible to find a different method, with respect to the FSQ-theorem, to relate slice monogenic functions and the monogenic functions.

\medskip
The FSQ mapping theorem in integral form, introduced in \cite{CoSaSo},
 is associated with the second step of the FSQ extension procedure.
In fact, the main idea is to apply the linear operator
$$
T_{FSQ2}=\Delta_{n+1}^{\frac{n-1}{2}}
$$
to the Cauchy kernel
$$
S_L^{-1}(s,x):=(s-\bar x)(s^2-2{\rm Re}(x)s+|x|^2)^{-1},
$$
of left slice monogenic functions when $n$ is odd, similarly we proceed for right slice monogenic functions.
For odd dimension applying the operator $\Delta_{n+1}^{\frac{n-1}{2}}$, in the variables in $x$, to the function $S_L^{-1}(s,x)$ we have obtained a very simple expression given by
\[
\begin{split}
\Delta_{n+1}^{\frac{n-1}{2}}S_L^{-1}(s,x)
=\gamma_n(s-\bar x)(s^2-2{\rm Re}(x)s +|x|^2)^{-\frac{n+1}{2}},
\end{split}
\]
where $\gamma_n$ are
\begin{equation}\label{gammn}
\gamma_n:= (-1)^{\frac{n-1}{2}} 2^{n-1} \left[\Gamma\left( \frac{n+1}{2} \right) \right]^2
\end{equation}
which can be used to obtain the Fueter-Sce mapping theorem in integral form.
Precisely, let $f$ be a slice monogenic function defined in an open set that contains $\overline{U}$, where $U$ is a
bounded  axially symmetric open set. Suppose that the boundary of $U\cap \mathbb{C}_I$  consists of a finite number of rectifiable Jordan curves for any $I\in\mathbb{S}$.
Then, if $x\in U$, the monogenic function $\breve{f}(x)$, given by
\begin{equation}\label{FSQdiffer}
\breve{f}(x)=\Delta_{n+1}^{\frac{n-1}{2}}f(x)
\end{equation}
 admits the integral representation
\begin{equation}\label{FueterGG}
 \breve{f}(x)=\frac{1}{2 \pi}\int_{\pp (U\cap \mathbb{C}_I)} \mathcal{F}_L(s,x)ds_I f(s),\ \ \ ds_I=ds/ I,
\end{equation}
where
$$
\mathcal{F}_L(s,x):=\gamma_n(s-\bar x)(s^2-2{\rm Re}(x)s +|x|^2)^{-\frac{n+1}{2}}
$$
is called the left $F$-kernel,
and the integral depends neither on $U$ nor on the  imaginary unit $I\in\mathbb{S}$.

The main problems studied in this paper can be formulated as follows.

\begin{problem}

(A) Determine the type of hyperholomorphicity of the map
 $$(s,x)\mapsto (s-\bar x)(s^2-2{\rm Re}(x)s +|x|^2)^{-h},
  $$
  for $h\in \mathbb{R}$, with respect to $s$ and $x$ for $s\not\in[x]$.

(B) Compute explicitly the Fourier transform of the slice monogenic Cauchy kernels and of the
$F_n$-kernels as functions of the Poisson kernel.

(C) Show that the relation
 $\Delta_{n+1}^{\frac{n-1}{2}}S_L^{-1}(s,x)
=\gamma_n(s-\bar x)(s^2-2{\rm Re}(x)s +|x|^2)^{-\frac{n+1}{2}}$ is true for all dimensions $n$ replacing suitably the constants $\gamma_n$.
\end{problem}

\medskip
The main results of this paper can now be summarized in the following steps.

(I)
It is a remarkable fact that  the function
$$
(s,x)\mapsto (s-\bar x)(s^2-2{\rm Re}(x)s +|x|^2)^{-h}
$$
is slice monogenic in $s$ for any $h\in\mathbb N$ but is monogenic in $x$ if and only if  $h=(n+1)/2$,
 namely if and only if $h$ is relate with the Sce's exponent also for $n$ odd.
 In Theorem \ref{sceqianexponen} we have shown that this
 result remains true also in the case of even dimension, that is when we have the fractional powers.

(II) The  Fourier transform, denoted by the symbol $\mathrm{F}$, of $S_L^{-1}(s, x)$ is
$$
\mathrm{F}[S_L^{-1}(s, \cdot)](\xi)
= c_n\  \frac{\bar{\xi}}{(\xi_0^2+| \underline{\xi}|^2)^{\frac{n+1}{2}}} \ e^{-is \xi_0}, \ \ \
\xi=\xi_0+\underline{\xi}\not=0,
$$
where $$c_n:=i 2^n \pi^{\frac{n+1}{2}} \Gamma \left( \frac{n+1}{2} \right).$$

(III)
A second fundamental result is the Fourier transform of the $F_n$-kernels.
We proved that
$$
\mathrm{F}[\mathcal{F}^L_n(s, \cdot)](\xi)
= k_n\  \frac{\bar{\xi}}{\xi_0^2+| \underline{\xi}|^2} \ e^{-is \xi_0},\ \ \ \
\xi_0+\underline{\xi}\not=0,
$$
where
$$
k_n:=i(-1)^{\frac{n-1}{2}}2^n  \pi^{\frac{n+1}{2}} \Gamma \left( \frac{n+1}{2} \right).
$$

(IV) We show that  the relation
$$
\Delta_{n+1}^{\frac{n-1}{2}}S_L^{-1}(s,x)
=\gamma_n(s-\bar x)(s^2-2{\rm Re}(x)s +|x|^2)^{-\frac{n+1}{2}}
$$
holds true also when $n$ is an  even number, using the Fourier transform of the kernels $S_L^{-1}$ and
$\mathcal{F}^L_n$.

\medskip
{\em The plan of the paper}.
The paper contains three sections including the introduction.
 In Section 2 we show the monogenicity of the Fueter-Sce kernel in even dimension.
Section 3 we compute the Fourier transform of the slice monogenic Cauchy kernels.
In Section 4 we compute
the Fourier transform of the $F_n$-kernels.
Finally in Section 5 we show that the relation
$$
\Delta_{n+1}^{(n-1)/2}S_L^{-1}(s,x)=\mathcal{F}_n^L(s,x),\ \ {\rm for}\ \ s\not\in[x]
$$
also holds for the even dimension of the Clifford algebra $\mathbb{R}_n$. The proof is based on the  the Fourier transform.

\section{Monogenicity of the Fueter-Sce kernel in even dimension}\label{preliminary}

In this paper we use the definition of  slice monogenic functions that is the generalization of
slice monogenic functions in the spirit of the FSQ mapping theorem and it is slightly different from the one in \cite{NONCOMMBOOK}.
This definition is the most appropriate for operator theory
and the reason is widely explained in several papers and in the books \cite{FRACTBOOK,6CKG}.
Keeping in mind the notations previously given for
the Clifford algebra $\mathbb{R}_n$ we recall some definitions.
For the missing proofs of the results that we recall see for example \cite{FCAL1,CoSaSo}.

\begin{definition}
Let $U\subseteq \mathbb{R}^{n+1}$.
 We say that $U$ is axially symmetric if $[x]\in U$ for every $x\in U$.
\end{definition}

\begin{definition}[Slice monogenic functions]\label{SHolDef}
 Let $U\subseteq\mathbb{R}^{n+1}$ be an axially symmetric open set and let $\mathcal{U} = \{ (u,v)\in\rr^2: u+ \mathbb{S} v\subset U\}$. A function $f:U\to \mathbb{R}_n$ is called a left
 slice function, if it is of the form
 \[
 f(x) = f_{0}(u,v) + If_{1}(u,v)\qquad \text{for } x = u + I v\in U
 \]
with the two functions $f_{0},f_{1}: \mathcal{U}\to \mathbb{R}_n$ that satisfy the compatibility conditions
\begin{equation}\label{CCondslic}
f_{0}(u,-v) = f_{0}(u,v),\qquad f_{1}(u,-v) = -f_{1}(u,v).
\end{equation}
If in addition $f_{0}$ and $f_{1}$ are $\mathcal C^1$ and satisfy the Cauchy-Riemann equations
\begin{equation}\label{CRR}
\begin{split}
&\partial_u f_{0}(u,v) -\partial_vf_{1}(u,v)=0
\\
&
\partial_v f_{0}(u,v) +\partial_uf_{1}(u,v)=0
\end{split}
\end{equation}
 then $f$ is called left slice monogenic.
A function $f:U\to \mathbb{R}_n$ is called a right slice function if it is of the form
\[
f(x) = f_{0}(u,v) + f_{1}(u,v) I\qquad \text{for } x = u+ Iv \in U
\]
with the two functions $f_{0},f_{1}: \mathcal{U}\to \mathbb{R}_n$ that satisfy (\ref{CCondslic}).
If $f_{0}$ and $f_{1}$ are $\mathcal C^1$ and satisfy the Cauchy-Riemann equations
(\ref{CRR})
 then $f$ is called right slice monogenic.

If $f$ is a left (or right) slice function such that $f_{0}$ and $f_{1}$ are real-valued, then $f$ is called intrinsic.

We denote the sets of left, right and intrinsic
 slice monogenic functions on $U$ by $\mathcal{SM}_L(U)$,
$\mathcal{SM}_R(U)$ and $\mathcal{N}(U)$, respectively.

\end{definition}

For slice monogenic functions we have two equivalent ways
to write the Cauchy kernels.

\begin{proposition}\label{secondAA}
If $x, s\in \mathbb{R}^{n+1}$ with $x\not\in [s]$, then
\begin{gather}\label{secondAAEQ}
-(x^2 -2x \Re  (s)+|s|^2)^{-1}(x-\overline s)=(s-\bar q)(s^2-2\Re (x)s+|x|^2)^{-1}
\end{gather}
and
\begin{gather}\label{secondAAEQ1}
 (s^2-2\Re (x)s+|x|^2)^{-1}(s-\bar x)=-(x-\bar s)(x^2-2\Re (s)x+|s|^2)^{-1} .
\end{gather}
\end{proposition}
So we can give the following definition to distinguish the two representations of the Cauchy kernels.
\begin{definition}\label{FORMSCK}
Let $x,s\in \mathbb{R}^{n+1}$ with $x\not\in [s]$.
\begin{itemize}
\item
We say that  $S_L^{-1}(s,x)$ is written in the form I if
$$
S_L^{-1}(s,x):=-(x^2 -2 \Re  (s) x+|s|^2)^{-1}(x-\overline s).
$$
\item
We say that $S_L^{-1}(s,x)$ is written in the form II if
$$
S_L^{-1}(s,x):=(s-\bar x)(s^2-2\Re (x) s+|x|^2)^{-1}.
$$
\item
We say that  $S_R^{-1}(s,x)$ is written in the form I if
$$
S_R^{-1}(s,x):=-(x-\bar s)(x^2-2\Re (s)x+|s|^2)^{-1} .
$$
\item
We say that $S_R^{-1}(s,x)$ is written in the form II if
$$
S_R^{-1}(s,x):=(s^2-2\Re (x)s+|x|^2)^{-1}(s-\bar x).
$$
\end{itemize}
\end{definition}
\begin{lemma} Let $x,s\in \mathbb{R}^{n+1}$ with $s\notin [x]$.
The left slice monogenic Cauchy kernel $S_L^{-1}(s,x)$ is left slice monogenic in $x$ and right slice monogenic in $s$.
The right slice monogenic Cauchy kernel $S_R^{-1}(s,x)$ is left slice monogenic in $s$ and right slice monogenic in $x$.
\end{lemma}

\begin{definition}[Slice Cauchy domain]\label{Slice Cauchy domain}
An axially symmetric open set $U\subset \mathbb{R}^{n+1}$ is called a slice Cauchy domain, if $U\cap\cc_I$ is a Cauchy domain in $\cc_I$ for any $I\in\mathbb{S}$. More precisely, $U$ is a slice Cauchy domain if, for any $I\in\mathbb{S}$, the boundary ${\partial( U\cap\cc_I)}$ of $U\cap\cc_I$ is the union a finite number of non-intersecting piecewise continuously differentiable Jordan curves in $\cc_{I}$.
\end{definition}
\begin{theorem}[Cauchy formulas]\label{Cauchy formulas}
\label{Cauchygenerale}
Let $U\subset\mathbb{R}^{n+1}$ be a slice Cauchy domain, let $I\in\mathbb{S}$ and set  $ds_I=ds (-I)$.
If $f$ is a (left) slice monogenic function on a set that contains $\overline{U}$ then
\begin{equation}\label{cauchynuovo}
 f(x)=\frac{1}{2 \pi}\int_{\partial (U\cap \mathbb{C}_I)} S_L^{-1}(s,x)\, ds_I\,  f(s),\qquad\text{for any }\ \  x\in U.
\end{equation}
If $f$ is a right slice monogenic function on a set that contains $\overline{U}$,
then
\begin{equation}\label{Cauchyright}
 f(x)=\frac{1}{2 \pi}\int_{\partial (U\cap \mathbb{C}_I)}  f(s)\, ds_I\, S_R^{-1}(s,x),\qquad\text{for any }\ \  x\in U.
 \end{equation}
These integrals  depend neither on $U$ nor on the imaginary unit $I\in\mathbb{S}$.
\end{theorem}
Even though $S_L^{-1}(s,x)$ written in the form I is more suitable for several applications, for example for the definition of a functional calculus, see \cite{NONCOMMBOOK}, it does not allow easy computations of the powers of the Laplacian
$$
\Delta_{n+1}=\pp^2_{x_0}+\sum_{j=1}^n\pp^2_{x_j},
$$ with respect to the variable $x$ applied to it.
The form II is the one that allows, by iteration, the computation of
$\Delta_{n+1}^{\frac{n-1}{2}} S_L^{-1}(s,x)$. In the sequel we will write $\Delta$ instead of $\Delta_{n+1}$.
\begin{theorem}\label{Laplacian_comp}
Let $x$,
$s\in \rr^{n+1}$
be such that
 $x\not\in [s]$.
Let
$
S_L^{-1}(s,x)=(s-\bar x)(s^2-2{\rm Re}(x)s+|x|^2)^{-1}
$
 be the slice-monogenic Cauchy kernel and let
 $
 \Delta=\sum_{i=0}^n\frac{\partial^2}{\partial x_i^2}
 $
  be the Laplace operator in the variable $x$.
Then, for $h\geq 1$, we have:
\begin{equation}\label{hLaplacian}
\Delta^hS_L^{-1}(s,x)=(-1)^h\prod_{\ell=1}^h(2\ell) \prod_{\ell=1}^h (n-(2\ell -1))
(s-\bar x)(s^2-2{\rm Re}(x)s +|x|^2)^{-(h+1)}
\end{equation}
and
\begin{equation}\label{hLaplacianR}
\Delta^hS_R^{-1}(s,x)=(-1)^h\prod_{\ell=1}^h(2\ell) \prod_{\ell=1}^h (n-(2\ell -1))(s^2-2{\rm Re}(x)s +|x|^2)^{-(h+1)}(s-\bar x).
\end{equation}
\end{theorem}

We recal the definition of Monogenic functions.
\begin{definition}[Monogenic functions] Let $U$ be an open set in $\mathbb{R}^{n+1}$.
A real differentiable function $f: U\to \mathbb{R}_n$ is left monogenic if
$$
Df(x):=\pp_{x_0}f(x)+\sum_{i=1}^n e_i\pp_{x_i}f(x)=0.
$$
It is right monogenic if
$$
f(x)D:=\pp_{x_0}f(x)+\sum_{i=1}^n \pp_{x_i}f(x)e_i=0.
$$
\end{definition}

In the sequel, we will also need the property of slice monogenicity of the functions $\Delta^{h}S_L^{-1}(s,x)$ and $\Delta^{h}S_L^{-1}(s,x)$ as shown in the following result:
\begin{proposition}\label{Laplacian_smonogenic}
Let $x$,
$s\in \rr^{n+1}$
be such that
 $x\not\in [s]$. Then we have.

 (I)
The function $\Delta^hS_L^{-1}(s,x)$
is a right slice monogenic function in the variable $s$ for all $h\geq 0$ and for all $x\not\in [s]$.

(II)
The function $\Delta^hS_R^{-1}(s,x)$
is a left slice monogenic function in the variable $s$ for all $h\geq 0$ and for all $x\not\in [s]$.
\end{proposition}
\begin{proposition}\label{Laplacian}
Let $n$ be an odd number and
let $x$,
$s\in \rr^{n+1}$
be such that
 $x\not\in [s]$.
Then the function $\Delta^{h}S_L^{-1}(s,x)$ is a left monogenic function in the variable $x$,
and the function  $\Delta^{h}S_R^{-1}(s,x)$ is a right monogenic function in the variable $x$, if and only if $h=\frac{n-1}{2}$.
\end{proposition}

\begin{definition}[The $\mathcal{F}_n$-kernels]\label{Fnker}

Let $n$ be an odd number. Let $x$,
$s\in \rr^{n+1}$.
We define, for $s\not\in[x]$, the left $\mathcal{F}^L_n$-kernel as
$$
\mathcal{F}^L_n(s,x):=\Delta^{\frac{n-1}{2}}S_L^{-1}(s,x)
=\gamma_n(s-\bar x)(s^2-2{\rm Re}(x)s +|x|^2)^{-\frac{n+1}{2}},
$$
and the right $\mathcal{F}^R_n$-kernel as
$$
\mathcal{F}^R_n(s,x):=\Delta^{\frac{n-1}{2}}S_R^{-1}(s,x)
=\gamma_n(s^2-2{\rm Re}(x)s +|x|^2)^{-\frac{n+1}{2}}(s-\bar x),
$$
where $\gamma_n$ are given by (\ref{gammn}).
\end{definition}

\begin{theorem}[The Fueter-Sce mapping theorem in integral form, see \cite{CoSaSo}]
Let $n$ be an odd number.
Let $U\subset\mathbb{R}^{n+1}$ be a slice Cauchy domain, let $I\in\mathbb{S}$ and set  $ds_I=ds (-I)$.

(I) Let $f$ is a left slice monogenic function on a set that contains $\overline{U}$, then
the monogenic function
$
\breve{f}(x)=\Delta^{\frac{n-1}{2}}f(x)
$
 admits the integral representation
\begin{equation}\label{FueterL}
 \breve{f}(x)=\frac{1}{2 \pi}\int_{\pp (U\cap \mathbb{C}_I)} \mathcal{F}^L_n(s,x)ds_I f(s), \ \ x\in U.
\end{equation}

(II)
 Let $f$ is a right slice monogenic function on a set that contains $\overline{U}$, then
the monogenic function
$
\breve{f}(x)=\Delta^{\frac{n-1}{2}}f(x)
$
 admits the integral representation
\begin{equation}\label{FueterR}
 \breve{f}(x)=\frac{1}{2 \pi}\int_{\pp (U\cap \mathbb{C}_I)} f(s) ds_I \mathcal{F}^R_n(s,x), \ \ x\in U.
\end{equation}
Moreover,  the integrals in (\ref{FueterL}) and (\ref{FueterR}) depend neither on  $U$ nor on the imaginary unit
$I\in\mathbb{S}$.
\end{theorem}

We now consider the $\mathcal{F}_n$-kernels in Definition \ref{Fnker} also for the case of the fractional powers.
We recall that given a paravector $y=u+I_yv \in \mathbb{R}^{n+1} \setminus (- \infty,0]$, for
$ \alpha \in \mathbb{R}$, we can define the fractional powers as
\begin{equation}
\label{FP}
y^{\alpha}:= e^{\alpha \log y}= e^{\alpha(\ln|y|+I \arg(y))},
\end{equation}
where $ \arg (y)= \arccos \frac{u}{|y|}.$
The definition is analogue for the quaternions and the fractional powers so defined are slice monogenic functions.
The first natural question that arise is formulated in the following problem.
\begin{problem}
 Proposition \ref{Laplacian} claims that for $n$ odd number the functions
 $$
 (s,x)\mapsto (s-\bar x)(s^2-2{\rm Re}(x)s +|x|^2)^{-(h+1)}
 $$
  and
 $$
 (s,x)\mapsto (s^2-2{\rm Re}(x)s +|x|^2)^{-(h+1)}(s-\bar x),
 $$
   for $x\not\in [s]$,
 are monogenic in the variable $x$ if and only of
 $h=\frac{n-1}{2}$. Is it still true in the case $n$ even?
\end{problem}
The answer is positive and is formulated in the following theorem.

\begin{theorem}\label{sceqianexponen}
\label{mono}
Let $ \lambda$ be a real number and $x,s \in \mathbb{R}^{n+1}$ be such that $x \notin [s]$.
Let us define:
$$
k_L(s,x):= (s- \bar{x}) (s^2-2x_0s+|x|^2)^{-\lambda}, \qquad \lambda \in \mathbb{R},
$$
and
$$
k_R(s,x):= (s^2-2x_0s+|x|^2)^{-\lambda}(s- \bar{x}) , \qquad \lambda \in \mathbb{R},
$$
 for
$$
s^2-2x_0s+|x|^2 \in \mathbb{R}^{n+1}\setminus (- \infty,0].
$$
Then, the function $k_L(s,x)$ is
left monogenic function in the variable $x$
and $k_R(s,x)$ right monogenic in the variable $x$ if and only if $\lambda= \frac{n+1}{2}$.
\end{theorem}
\begin{proof}
We give the details for $k_L(s,x)$, similarly we proceed for $k_R(s,x)$. For simplicity, in the proof, we write
$k(s,x)$ for $k_L(s,x)$.
We have to compute $(\partial_{x_0}+ \partial_{\underline{x}})k(s,x)$, where $ \partial_{\underline{x}}= \sum_{j=1}^n e_{j}\partial_{x_j}.$ First, we put $s=u+Iv$, thus
\begin{eqnarray*}
s^2-2x_0s+|x|^2= (u^2-v^2-2x_0u+|x|^2)+I(2uv-2x_0v).
\end{eqnarray*}
Using the formula of fractional powers, in \eqref{FP}, we get
\begin{equation}
k(s,x)= (s- \bar{x}) e^{\alpha(u,v)},
\end{equation}
where
\begin{eqnarray*}
\alpha(u,v) &:=& - \frac{\lambda}{2} \ln[(u^2-v^2-2x_0u+|x|^2)^2+(2uv-2x_0v)^2]+\\
&& - \lambda I \arccos \frac{u^2-v^2-2x_0u+|x|^2}{ \sqrt{(u^2-v^2-2x_0u+|x|^2)^2+(2uv-2x_0v)^2}}.
\end{eqnarray*}
Let us denote
$$
\beta(u,v):=\frac{u^2-v^2-2x_0u+|x|^2}{ \sqrt{(u^2-v^2-2x_0u+|x|^2)^2+(2uv-2x_0v)^2}}.
$$
So we have
$$ k(s,x)=(s-\bar{x})e^{- \frac{\lambda}{2} \ln[(u^2-v^2-2x_0u+|x|^2)^2+(2uv-2x_0v)^2]- \lambda I \arccos \beta(u,v)}.$$
To compute $ \partial_{ x_0} k(s,x)$
we calculate the derivative of $ \beta(u,v)$ with respect to $x_0$
$$ \frac{\partial \beta(u,v)}{\partial x_0}= \frac{(-2u+2x_0)(2uv-2x_0v)^2+(u^2-v^2-2x_0u+|x|^2)(2uv-2x_0v)2v}{[(u^2-v^2-2x_0u+|x|^2)^2+(2uv-2x_0v)^2]^{3/2}}.$$
Thus
\begin{eqnarray*}
\frac{\partial  \arccos\beta(u,v)}{\partial x_0}&=&- \frac{\partial_{x_0} \beta(u,v)}{\sqrt{1- \beta^2(u,v)}}\\
&=&- \frac{(-2u+2x_0)(2uv-2x_0 v)+(u^2-v^2-2x_0u+|x|^2) 2v}{(u^2-v^2-2x_0u+|x|^2)^2+(2uv-2x_0v)^2}.
\end{eqnarray*}
So we have
\begin{eqnarray}
\label{P1}
\nonumber
\frac{\partial k(s,x)}{\partial x_0}& =& - e^{\alpha(u,v)} + \frac{(u+Iv- \bar{x}) 2 \lambda \Bigl[(u^2-v^2-2x_0u+|x|^2)(u-x_0+Iv)}{(u^2-v^2-2x_0u+|x|^2)^2+(2uv-2x_0v)^2}\\
&&
-\frac{I(2uv-2x_0v)(Iv+u-x_0)\Bigl]}{(u^2-v^2-2x_0u+|x|^2)^2+(2uv-2x_0v)^2}e^{\alpha(u,v)}.
\end{eqnarray}
Now, we compute the derivative of $k(s,x)$ with respect to $x_j$, $1\leq j \leq n$. As before we start from the derivative of $ \beta(u,v)$
$$ \frac{\partial \beta(u,v)}{\partial x_j}= \frac{2 x_j(2uv-2x_0v)^2}{[(u^2-v^2-2x_0u+|x|^2)^2+(2uv-2x_0v)^2]^{3/2}}.$$
Thus
$$\frac{\partial \arccos \beta(u,v) }{\partial x_j}=-\frac{2 x_j(2uv-2x_0v)}{(u^2-v^2-2x_0u+|x|^2)^2+(2uv-2x_0v)^2}.$$
Therefore
\begin{eqnarray*}
\frac{\partial k(s,x)}{\partial x_j}&=&  \! \! \! e_j e^{\alpha(u,v)} \!- \! \frac{(u+Iv- \bar{x})2 \lambda \Bigl[(u^2-v^2-2x_0u+|x|^2)-I(2uv-2x_0v) \Bigl]}{(u^2-v^2-2x_0u+|x|^2)^2+(2uv-2x_0v)^2} x_j e^{\alpha(u,v)}.
\end{eqnarray*}
Now, we are ready to compute $ \partial_{\underline{x}} k(s,x)$
\begin{eqnarray*}
\partial_{\underline{x}} k(s,x) &=& \sum_{j=1}^n e_j \frac{\partial k(s,x)}{\partial x_j}\\
&=&-n e^{\alpha(u,v)}- \frac{(u+Iv- \bar{x})2 \lambda \Bigl[(u^2-v^2-2x_0u+|x|^2)-I(2uv-2x_0v)\Bigl]}{(u^2-v^2-2x_0u+|x|^2)^2+(2uv-2x_0v)^2} \cdot\\
&& \cdot \left( \sum_{j=1}^n e_j x_j \right) e^{\alpha(u,v)}
\end{eqnarray*}
This implies that
\begin{equation}
\label{P2}
\! \!  \! \! \! \! \! \! \! \!\partial_{\underline{x}} k(s,x) \! \! = \! \!-n e^{\alpha(u,v)} \!- \! \frac{(u+Iv- \bar{x})2 \lambda \Bigl[(u^2-v^2-2x_0u+|x|^2)\! \!- \! \!I(2uv-2x_0v)\Bigl]}{(u^2-v^2-2x_0u+|x|^2)^2+(2uv-2x_0v)^2}  \underline{x} e^{\alpha(u,v)}.
\end{equation}
Hence form \eqref{P1} and \eqref{P2} we get
\begin{eqnarray*}
\! \! \! \!(\partial_{x_0}+ \partial_{\underline{x}})k(s,x) \! \! \! \!&=&  -(n+1) e^{\alpha(u,v)}+ \frac{2 \lambda (u+Iv- \bar{x})(u+Iv-x)\Bigl[(u^2-v^2-2x_0u+|x|^2)}{(u^2-v^2-2x_0u+|x|^2)^2+(2uv-2x_0v)^2} \\
&& -\frac{I(2uv-2x_0v)\Bigl]}{(u^2-v^2-2x_0u+|x|^2)^2+(2uv-2x_0v)^2}e^{\alpha(u,v)}.
\end{eqnarray*}
Now, we observe that
\begin{eqnarray*}
(u+Iv- \bar{x})(u+Iv-x)
&=& (u^2-v^2-2u x_0+|x|^2)+I(2uv-2x_0v).
\end{eqnarray*}
Setting
$$ \gamma(u,v):=(u^2-v^2-2u x_0+|x|^2)+I(2uv-2x_0v)$$
we therefore obtain
$$ (\partial_{x_0}+ \partial_{\underline{x}})k(s,x)= \left[-(n+1)+ 2 \lambda \frac{\gamma(u,v)  \cdot \overline{\gamma(u,v)}}{| \gamma(u,v)|^2}\right] e^{\alpha(u,v)}=\left[-(n+1)+ 2 \lambda \right] e^{\alpha(u,v)},$$
so we finally get
$$(\partial_{x_0}+ \partial_{\underline{x}})k(s,x)=0$$ if and only if $ \lambda= \frac{n+1}{2}$.
\end{proof}

\section{The Fourier transform of the slice monogenic Cauchy kernels.}

The main result of this section is the explicit computation of the Fourier transform of the slice monogenic Cauchy kernels  $S_L^{-1}(s,x)$
  and $S_R^{-1}(s,x)$ with respect to $x$ when $s$ is a real number.
  Then by extension we get the Fourier transform when $s$ is a paravector.
  Firstly, let us introduce the definition of Fourier transform that we will use.
\begin{definition}
\label{Fourier}
Let $f \in \mathcal{S}(\mathbb{R}^{n+1})$. The Fourier transform of the function $f$ is
$$ \hat{f}(\xi):= \mathrm{F}[f(x)](\xi)= \int_{\mathbb{R}^{n+1}} f(x) e^{-i(x, \xi)} dx,$$
where
$$(x, \xi)= \sum_{j=0}^n x_j \xi_j.$$
\end{definition}
\begin{definition}
\label{Fourier2}
We define the inverse Fourier transform of the function $f$ in the following way
$$ \mathcal{R}[{f}(\xi)](x)= \color{black}{\frac{1}{(2 \pi)^{n+1}}}\int_{\mathbb{R}^{n+1}} {f}(\xi) e^{i(x, \xi)} d \xi.$$
\end{definition}
In this paper we will use the following important result
\color{black}{ \begin{theorem}[Plancherel's Theorem]
\label{Plan}
If $ f,g \in \mathcal{S}(\mathbb{R}^{n+1})$ then
\begin{equation}
\int_{\mathbb{R}^{n+1}} f(x) \overline{g(x)} dx= \frac{1}{(2 \pi)^{n+1}}\int_{\mathbb{R}^{n+1}} \mathrm{F}(f)(\xi) \overline{\mathrm{F}(g(\xi))} d\xi.
\end{equation}
\end{theorem}
}
\begin{remark}
Using the Plancherel's theorem it is possible to obtain
\begin{equation}
\label{new3}
\int_{\mathbb{R}^{n+1}} \mathcal{R}(f)(x) \overline{\mathcal{R}(g(x))} dx= \frac{1}{(2 \pi)^{n+1}}\int_{\mathbb{R}^{n+1}} f(\xi) \overline{g(\xi)} d \xi. .
\end{equation}
\end{remark}
In the sequel we will need this result.
\begin{theorem} \cite[Sect. B.5]{G}
\label{rad}
Let $f(|\underline{x}|)$ be a radial function in $\color{black}{\mathcal{S}( \mathbb{R}^n)}$ with $n \geq 2$. Then the Fourier transform of $f$ is also radial and has the form
$$ \hat{f}(|\underline{\xi}|)=(2 \pi)^{\frac{n}{2}} | \underline{\xi}|^{- \frac{n-2}{2}} \int_{0}^{\infty} J_{\frac{n-2}{2}}(|\underline{\xi}| r) r^{\frac{n}{2}} f(r) \, dr,$$
where $r=|\underline{x}|$ and $ J_{\frac{n-2}{2}}$ are the Bessel functions.
\end{theorem}
\color{black}{\begin{remark}(See \cite{O1})
The formula in Theorem \ref{rad} is also valid for all functions
$$ f \in L^1(\mathbb{R}^{n}) \cap L^{2}(\mathbb{R}^{n}).$$
\end{remark}}
\color{black}{ Now, we start our computations.
\begin{theorem}
\label{TCR1}
Let us assume $s_0 \in \mathbb{R}$ and $x \in \mathbb{R}^{n+1}$. If we consider the slice monogenic Cauchy kernels $S_L^{-1}(s_0,x)$ and $S_R^{-1}(s_0,x)$ written in form II (see Definition \ref{FORMSCK}) then their Fourier transforms with respect to $x$ are equal and given by
$$
\mathrm{F}[S_L^{-1}(s, \cdot)](\xi)=\mathrm{F}[S_R^{-1}(s_0, \cdot)](\xi)= c_n \frac{\bar{\xi}}{(\xi_0^2+| \underline{\xi}|^2)^{\frac{n+1}{2}}} e^{-is_0 \xi_0}
,\ \ \ \
\xi_0+\underline{\xi}\not=0
$$
where 
$$
c_n:=i 2^n \pi^{\frac{n+1}{2}} \Gamma \left( \frac{n+1}{2} \right).
$$
Moreover, if $s=s_0+\underline{s}\in \mathbb{R}^{n+1}$ is a paravector the term $e^{-is_0 \xi_0}$ extends to the intrinsic entire slice monogenic function $e^{-is \xi_0}$
and we have the Fourier transforms of the Cauchy kernels:
\begin{equation}
\label{new1}
\mathrm{F}[S_L^{-1}(s, \cdot)](\xi)
= c_n\  \frac{\bar{\xi}}{(\xi_0^2+| \underline{\xi}|^2)^{\frac{n+1}{2}}} \ e^{-is \xi_0},\ \ \ \
\xi_0+\underline{\xi}\not=0
\end{equation}
and
\begin{equation}
\label{new2}
\mathrm{F}[S_R^{-1}(s, \cdot)](\xi)= c_n\ e^{-is \xi_0}\ \frac{\bar{\xi}}{(\xi_0^2+| \underline{\xi}|^2)^{\frac{n+1}{2}}},\ \ \ \
\xi_0+\underline{\xi}\not=0.
\end{equation}
The extension $\mathrm{F}[S_L^{-1}(s, \cdot)](\xi)$ is right slice monogenic in $s$, while
$\mathrm{F}[S_R^{-1}(s, \cdot)](\xi)$ is right slice monogenic in $s$.
\end{theorem}
\begin{proof}
In the following proof we always work with  $s=s_0\in \mathbb{R}$ since we have
$$
S_L^{-1}(s,x)=(s-\bar x)(s^2-2\Re (x) s+|x|^2)^{-1}=(s^2-2\Re (x)s+|x|^2)^{-1}(s-\bar x)
=S_R^{-1}(s,x).
$$
The extension from $s=s_0$ to $s=s_0+\underline{s}$ is immediate.
So in our computations, we set
$$
S^{-1}(s,x):=S_L^{-1}(s,x)=S_R^{-1}(s,x)=\color{black}{\frac{s-x_0+ \underline{x}}{(s-x_0)^2+| \underline{x}|^2}}, \ \ \ s\in \mathbb{R}.
$$
We put $x= x_0 +\underline{x}$ and we recall the identification of the paravectors with
$(x_0,...,x_n)$.
\color{black}{Since the function $ S^{-1}(s,x)$ is not in $L^{1}(\mathbb{R}^{n+1})$ we have to perform the computations in the distributional sense. Firstly, we consider the following function
$$ f_{\underline{x}}(x_0):= \frac{s-x_0}{(s-x_0)^2+| \underline{x}|^2}.$$
Let $ \varphi \in \mathcal{S}(\mathbb{R}^{n+1})$. Then we have
\begin{eqnarray}
\label{Tao1}
\nonumber
\int_{\mathbb{R}^{n+1}} f_{\underline{x}}(x_0)\overline{\mathrm{F}(\varphi)(x)} dx &=& \int_{\mathbb{R}^{n+1}} f_{\underline{x}}(x_0)\overline{\mathrm{F}_0\left( \mathrm{F}_n \varphi_{\underline{x}}\right)(x_0)} d x_0 d \underline{x}\\
&=& \int_{\mathbb{R}^n} \int_{\mathbb{R}} \left( \mathrm{F}_0 f_{\underline{x}}\right)(\xi_0)\overline{\left( \mathrm{F}_n \varphi_{\underline{x}}\right)(\xi_0)} d \xi_0 d \underline{x},
\end{eqnarray}
where $ d \underline{x}= dx_1...dx_n$, $ \mathrm{F}_{0}$ is the Fourier transform with respect to the variable $x_0$ and $\mathrm{F}_{n}$ is the Fourier transform with respect to the other variables. Now, we compute $\mathrm{F}_0 f_{\underline{x}}(\xi_0)$. First of all we make the following change of variables $s+y=x_0$, thus by basic properties of the Fourier transform we have
\begin{eqnarray}
\label{Tao2}
\nonumber
\mathrm{F}_0 f_{\underline{x}}(\xi_0) &=& -\mathrm{F}_y[y(y^2+|\underline{x}|^2)^{-1}](\xi_0)e^{-is \xi_0}= -i \frac{d}{d \xi_0} \mathrm{F}_y \left( \frac{1}{|\underline{x}|^2+y^2} \right) (\xi_0)e^{-is \xi_0}\\
&=&-i \frac{\pi}{|\underline{x}|} \left( \frac{d}{d \xi_0} e^{-|\underline{x}| | \xi_0|} \right) e^{-is \xi_0}= i\frac{ \pi \xi_0}{| \xi_0|} e^{-|\underline{x}|| \xi_0|}  e^{-is\xi_{0}} .
\end{eqnarray}
Since $ \varphi_{\underline{x}}(x_0)=\varphi(x)$ and by Fubini's theorem we have
\begin{eqnarray*}
\int_{\mathbb{R}^{n+1}} f_{\underline{x}}(x_0)\overline{\mathrm{F}(\varphi)(x)} dx &=& \int_{\mathbb{R}} \int_{\mathbb{R}^n} \mathrm{F}_n \left( \mathrm{F}_0 f_{\underline{x}} \right) (\xi_0) \overline{\varphi(\xi)} d \underline{\xi} d \xi_0\\
&=& \int_{\mathbb{R}^{n+1}}\mathrm{F}_n \left( \mathrm{F}_0 f_{\underline{x}} \right) (\xi_0) \overline{\varphi(\xi)} d \xi.
\end{eqnarray*}
We finish this first part by computing $\mathrm{F}_n \left( \mathrm{F}_0 f_{\underline{x}} \right) (\xi_0)$. By Theorem \ref{rad} with $r = | \underline{x}|$ we have
$$ \mathrm{F}_n \left( \mathrm{F}_0 f_{\underline{x}} \right) (\xi_0)=(2 \pi)^{\frac{n}{2}}  | \underline{\xi}|^{- \frac{n-2}{2}} i \frac{ \pi \xi_0}{| \xi_0|}  e^{-is\xi_0} \int_0^\infty J_{\frac{n-2}{2}}( | \underline{\xi}|r) r^{\frac{n}{2}} e^{-r | \xi_0|} dr.$$
Now, we make another change of variables $t=| \underline{\xi}|r$.
$$
\mathrm{F}_n \left( \mathrm{F}_0 f_{\underline{x}} \right) (\xi_0)=(2 \pi)^{\frac{n}{2}}   | \underline{\xi}|^{-n} i \frac{\pi \xi_0}{| \xi_0|} e^{-is\xi_0} \int_0^\infty J_{\frac{n-2}{2}}( t) t^{\frac{n}{2}} e^{- \frac{t | \xi_0|}{| \underline{\xi}|}} dt.
$$
From \cite[formula 6.623(2)]{GR} we know that
$$
 \int_{0}^\infty e^{-at} t^{\nu+1} J_{\nu}(bt) dt= \frac{2a(2b)^{\nu} \Gamma \left( \nu+ \frac{3}{2} \right)}{\sqrt{\pi}(a^2+b^2)^{\nu + \frac{3}{2}}},\qquad  \nu >-1, \quad a>0, \quad b>0.
$$
In our case $a:= \frac{| \xi_0|}{| \underline{\xi}|}$, $ \nu:= \frac{n}{2}-1$, $ b:=1$. Since $n \geq 2$ all conditions on the parameters are satisfied.
Thus, we have
\begin{eqnarray*}
\mathrm{F}_n \left( \mathrm{F}_0 f_{\underline{x}} \right) (\xi_0) \! \! &=& \! \! \! \! \frac{i}{2}(2 \pi)^{\frac{n}{2}+1} \frac{\xi_0}{|\xi_0|}| \underline{\xi}|^{-n} e^{-is\xi_0}2 \frac{| \xi_0|}{|\underline{\xi}|}2^{\frac{n}{2}-1} \frac{\Gamma\left(\frac{n+1}{2} \right)}{\sqrt{\pi} \left( \frac{\xi_0^2}{| \underline{\xi}|^2}+1 \right)^{\frac{n+1}{2}}}\\
&=& \frac{i 2^n \pi^{\frac{n+1}{2}} \xi_0 |\underline{\xi}|^{-n-1}e^{-is \xi_0} |\underline{\xi}|^{n+1}\Gamma\left(\frac{n+1}{2} \right)}{(\xi^2_0+| \underline{\xi}|^2)^{\frac{n+1}{2}}}\\
&=& \frac{i 2^n \pi^{\frac{n+1}{2}} \Gamma\left(\frac{n+1}{2} \right) \xi_0  e^{-is \xi_0} }{(\xi^2_0+| \underline{\xi}|^2)^{\frac{n+1}{2}}}.
\end{eqnarray*}
Hence
\begin{equation}
\label{CR1}
\int_{\mathbb{R}^{n+1}} f_{\underline{x}}(x_0)\overline{\mathrm{F}(\varphi)(x)} dx= c_n \int_{\mathbb{R}^{n+1}} \frac{\xi_0}{(\xi^2_0+| \underline{\xi}|^2)^{\frac{n+1}{2}}}  e^{-is \xi_0} \overline{\varphi}(\xi) d \xi,
\end{equation}
where
$
c_n:=i 2^n \pi^{\frac{n+1}{2}} \Gamma \left( \frac{n+1}{2} \right).
$
Now, we compute the Fourier transform of
$$
h_{\underline{x}}(x_0):= \frac{\underline{x}}{(s-x_0)^2+| \underline{x}|^2} = \sum_{j=1}^n e_j x_j u_{\underline{x}}(x_0),
$$
where we have set
$$
u_{\underline{x}}(x_{0}):=\frac{1}{{(s-x_0)^2+| \underline{x}|^2}}.
$$
Let $ \varphi \in \mathcal{S}(\mathbb{R}^{n+1})$
\begin{eqnarray}
\int_{\mathbb{R}^{n+1}} h_{\underline{x}}(x_0) \overline{\mathrm{F}(\varphi)(x)} dx &=& \int_{\mathbb{R}^{n+1}} \sum_{j=1}^n e_j x_j u_{\underline{x}}(x_0) \overline{\mathrm{F}_0 \left( \mathrm{F}_n(\varphi_{\underline{x}}\right))(x_0)} d x_0 d \underline{x}\\
&=& \int_{\mathbb{R}^n} \int_{\mathbb{R}} \sum_{j=1}^n e_j x_j  \mathrm{F}_0(u_{\underline{x}})(\xi_0) \overline{\mathrm{F}_n \varphi_{\underline{x}}(\xi_0)} d \xi_0 d \underline{x}.
\end{eqnarray}
Now, we compute $\mathrm{F}_0(u_{\underline{x}})(\xi_0)$ using the following change of variable $s+y=x_0$
$$ \mathrm{F}_0(u_{\underline{x}})(\xi_0)= \mathrm{F}_{y} \left( \frac{1}{y^2+|\underline{x}|^2} \right)(\xi_0)e^{- is \xi_0}= \frac{\pi}{|\underline{x}|} e^{-|\underline{x}| | \xi_0|} e^{-is\xi_{0}}.
$$
By Fubini's theorem and the fact that $ \varphi_{\underline{x}}(x_0)= \varphi(x)$ we have
$$
 \int_{\mathbb{R}^{n+1}} h_{\underline{x}}(x_0) \overline{\mathrm{F}(\varphi)(x)} dx=\int_{\mathbb{R}} \int_{\mathbb{R}^n} \sum_{j=1}^n e_j \mathrm{F}_n \left( x_j
 \mathrm{F}_0(u_{\underline{x}}) \right) (\xi_0) \overline{ \varphi(\xi)}  d \underline{\xi} d \xi_0.
 $$
From the following basic property of the Fourier transform
$$
\mathrm{F}[xf(x)](\xi)= i \frac{d}{ d \xi} ( \mathrm{F} f(x))(\xi),\ \ \ x,  \xi\in \mathbb{R}
$$
 we get
$$
\int_{\mathbb{R}^{n+1}} h_{\underline{x}}(x_0)
\overline{\mathrm{F}(\varphi)(x)} dx=i \int_{\mathbb{R}} \int_{\mathbb{R}^n} \sum_{j=1}^n e_j \frac{\partial}{\partial \xi_j} \mathrm{F}_n \left(   \mathrm{F}_0(u_{\underline{x}}) \right) (\xi_0) \overline{ \varphi(\xi)}  d \underline{\xi} d \xi_0.
$$
We complete the proof of this theorem by computing $\mathrm{F}_n \left(   \mathrm{F}_0(u_{\underline{x}}) \right) (\xi_0) $. By Theorem \ref{rad}, with $r= | \underline{x}|$, we get
$$ \mathrm{F}_n \left(  \mathrm{F}_0(u_{\underline{x}}) \right) (\xi_0)=(2 \pi)^{\frac{n}{2}} \pi e^{-is \xi_0}  | \underline{\xi} |^{- \frac{n-2}{2}} \int_{0}^\infty J_{\frac{n-2}{2}}(| \underline{\xi}| r) r^{\frac{n}{2}-1} e^{-r| \xi_0|}    dr .$$
Now we put $t= r| \underline{\xi}|$
$$
\mathrm{F}_n   (\mathrm{F}_0(u_{\underline{x}})) (\xi_0)= \frac{1}{2}(2 \pi)^{\frac{n}{2}+1} e^{-is \xi_0}   | \underline{\xi} |^{1-n} \int_{0}^\infty J_{\frac{n-2}{2}}(t) t^{\frac{n}{2}-1} e^{- t \frac{| \xi_0|}{| \underline{\xi}|}}  dt .
$$
From \cite[formula 6.623 (1)]{GR} we know that
$$ \int_{0}^{\infty} e^{-at} t^{\nu} J_{\nu}(bt) dt= \frac{(2b)^\nu \Gamma \left( \nu + \frac{1}{2}\right)}{\sqrt{\pi}(a^2+b^2)^{\nu+ \frac{1}{2}}},\qquad \nu > - \frac{1}{2}, \quad a>0, \quad b>0.$$
Thus by putting $b:=1$, $ \nu:= \frac{n}{2}-1$ and $a:= \frac{| \xi_0|}{| \underline{\xi}|}$ we obtain
\begin{eqnarray*}
\label{CR2}
\nonumber
\mathrm{F}_n \left(   \mathrm{F}_0(u_{\underline{x}}) \right) (\xi_0) &=& \frac{1}{2}(2 \pi)^{\frac{n}{2}+1}  e^{-is \xi_0}  \left( \frac{|\underline{\xi}|^{1-n} 2^{\frac{n}{2}-1}\Gamma \left( \frac{n-1}{2} \right)}{\sqrt{\pi} \left( \frac{\xi_0^2}{|\underline{\xi}|^2}+1 \right)^{\frac{n-1}{2}}} \right)\\ \nonumber
&=&  2^{n-1} \pi^{\frac{n+1}{2}} e^{-is \xi_0} \Gamma \left( \frac{n-1}{2} \right) \left( \frac{|\underline{\xi}|^{1-n} |\underline{\xi}|^{n-1} }{(\xi_0^2+|\underline{\xi}|^2)^{\frac{n-1}{2}}} \right)\\
&=&  2^{n-1} \pi^{\frac{n+1}{2}} e^{-is \xi_0} \Gamma \left( \frac{n-1}{2} \right)\left( \frac{1}{(\xi_0^2+|\underline{\xi}|^2)^{\frac{n-1}{2}}} \right).
\end{eqnarray*}
We compute the derivative
\begin{equation}
\label{der}
\frac{\partial}{\partial \xi_j} \left( \frac{1}{(\xi_0^2+|\underline{\xi}|^2)^{\frac{n-1}{2}}} \right)= -\frac{\left( \frac{n-1}{2} \right) (\xi^2+| \underline{\xi}|^{2})^{\frac{n-3}{2}}2 \xi_j}{(\xi_0^2+|\underline{\xi}|^{2})^{n-1}}=- \frac{2 \xi_j \left( \frac{n-1}{2}\right)}{(\xi_0^2+|\underline{\xi}|^{2})^{\frac{n+1}{2}}}.
\end{equation}
Now, by using the following property of the Gamma function $ \Gamma(x+1)=x \Gamma(x)$, for $x>0$, we obtain
\begin{eqnarray*}
\sum_{j=1}^n e_j \frac{\partial}{\partial \xi_j} \mathrm{F}_n \left(   \mathrm{F}_0(u_{\underline{x}}) \right) (\xi_0) &=& - \frac{2^{n-1} \pi^{\frac{n+1}{2}} e^{-is \xi_0} \left( \frac{n-1}{2}  \right)\Gamma \left( \frac{n-1}{2} \right) 2 \sum_{j=1}^n e_j \xi_j}{(\xi_0^2+|\underline{\xi}|^{2})^{\frac{n+1}{2}}}\\
&=& - \frac{2^n \pi^{\frac{n+1}{2}} \Gamma \left( \frac{n+1}{2} \right) }{(\xi_0^2+|\underline{\xi}|^{2})^{\frac{n+1}{2}}} \underline{\xi}e^{-is \xi_0}.
\end{eqnarray*}
Hence we get
\begin{equation}
\label{CR3}
\int_{\mathbb{R}^{n+1}} h_{\underline{x}}(x_0) \overline{\mathrm{F}(\varphi)(x)} dx=-c_n \int_{\mathbb{R}^{n+1}} \frac{\underline{\xi}}{{(\xi_0^2+|\underline{\xi}|^{2})^{\frac{n+1}{2}}}}e^{-is \xi_0} \overline{ \varphi(\xi)}  d \xi ,
\end{equation}
where $c_n:=i 2^n \pi^{\frac{n+1}{2}} \Gamma \left( \frac{n+1}{2} \right) $.
Finally from \eqref{CR1} and \eqref{CR3} we get
\begin{eqnarray*}
\int_{\mathbb{R}^{n+1}} S^{-1}(s, x) \overline{\mathrm{F}(\varphi)(x)} dx &=&  \int_{\mathbb{R}^{n+1}} f_{\underline{x}}(x_0) \overline{\mathrm{F}(\varphi)(x)} dx-\int_{\mathbb{R}^{n+1}} h_{\underline{x}}(x_0) \overline{\mathrm{F}(\varphi)(x)} dx\\
&=& c_n \int_{\mathbb{R}^{n+1}}\frac{\bar{\xi}}{(\xi_0^2+| \underline{\xi}|^2)^{\frac{n+1}{2}}} e^{-is \xi_0}\overline{ \varphi(\xi)}  d \xi.
\end{eqnarray*}
This proves \eqref{new1} and \eqref{new2}, respectively.}

\medskip
\color{black}{We are now in the position
to observe that the term $e^{-is \xi_0}$, for $s=s_0$ extends to the entire intrinsic slice monogenic function $e^{-i(s_0+\underline{s}) \xi_0}$. So the function
$$
s_0\mapsto c_n \frac{\bar{\xi}}{(\xi_0^2+| \underline{\xi}|^2)^{\frac{n+1}{2}}} e^{-is_0 \xi_0}
$$
has the right slice monogenic extension in $s\in \mathbb{R}^{n+1}$
$$
\mathrm{F}[S_L^{-1}(s, \cdot)](\xi)
= c_n\  \frac{\bar{\xi}}{(\xi_0^2+| \underline{\xi}|^2)^{\frac{n+1}{2}}} \ e^{-is \xi_0}
$$
while the  right slice monogenic extension in $s\in \mathbb{R}^{n+1}$ is
$$
\mathrm{F}[S_R^{-1}(s, \cdot)](\xi)= c_n\ e^{-is \xi_0}\ \frac{\bar{\xi}}{(\xi_0^2+| \underline{\xi}|^2)^{\frac{n+1}{2}}}
$$
and this conclude the proof.}
\end{proof}

\section{The Fourier transform of the $F_n$-kernels}

Thanks to Theorem \ref{sceqianexponen} the $F_n$-kernels are
meaningful also with $n$ odd where we interpret the fractional powers of paravectors are
slice monogenic functions.

As we have shown when  $n$ be an odd number and $x$, $s\in \rr^{n+1}$,
for $s\not\in[x]$, the relations
$$
\Delta^{\frac{n-1}{2}}S_L^{-1}(s,x)
=\gamma_n(s-\bar x)(s^2-2{\rm Re}(x)s +|x|^2)^{-\frac{n+1}{2}},
$$
and
$$
\Delta^{\frac{n-1}{2}}S_R^{-1}(s,x)
=\gamma_n(s^2-2{\rm Re}(x)s +|x|^2)^{-\frac{n+1}{2}}(s-\bar x),
$$
where $\gamma_n$ are given by (\ref{gammn}) are obtained by a long, but direct computation.
We now let $n$ be any natural number and when $n$ is even we interpret the terms
$(s^2-2{\rm Re}(x)s +|x|^2)^{-\frac{n+1}{2}}$  as fractional power for $x$,
$s\in \rr^{n+1}$.
We define, for $s\not\in[x]$, the left $\mathcal{F}^L_n$-kernel as
$$
\mathcal{F}^L_n(s,x):
=\gamma_n(s-\bar x)(s^2-2{\rm Re}(x)s +|x|^2)^{-\frac{n+1}{2}},
$$
and the right $\mathcal{F}^R_n$-kernel as
$$
\mathcal{F}^R_n(s,x):
=\gamma_n(s^2-2{\rm Re}(x)s +|x|^2)^{-\frac{n+1}{2}}(s-\bar x),
$$
where $\gamma_n$, are given by (\ref{gammn}),
are now interpreted in terms of the Euler's Gamma function
$$
\gamma_n:= (-1)^{\frac{n-1}{2}} 2^{n-1} \left[\Gamma\left( \frac{n+1}{2} \right) \right]^2.
$$
We will compute the Fourier transforms of $\mathcal{F}_L(s,x)$ and $\mathcal{F}_R(s,x)$ with respect to $x$.

\begin{remark} Observe that when $s=s_0\in \mathbb{R}$, for $s\not\in[x]$, then we have
\[
\begin{split}
\mathcal{F}^L_n(s_0,x)
&=\gamma_n(s_0-\bar x)(s_0^2-2{\rm Re}(x)s_0 +|x|^2)^{-\frac{n+1}{2}}
\\
&
=\gamma_n(s_0^2-2{\rm Re}(x)s_0 +|x|^2)^{-\frac{n+1}{2}}(s_0-\bar x)
\\
&
=\mathcal{F}^R_n(s_0,x).
\end{split}
\]
So for simplicity in the following when $s=s_0\in \mathbb{R}$ we use the notation
$$
\mathcal{F}_n(s_0,x):=\mathcal{F}^L_n(s_0,x)=\mathcal{F}^R_n(s_0,x).
$$
\end{remark}

\begin{theorem}
\label{TCR2}
Let us assume $ x \in \mathbb{R}^{n+1}$ and $s$ a real number. The Fourier transform of $\mathcal{F}_n(s,x)$ with respect to $x$ is
$$\widehat{\mathcal{F}_n}(s,\xi)=k_n \frac{\bar{\xi}}{\xi^2_0+|\underline{\xi}|^2} e^{-is \xi_0},$$
where
$$
k_n:=i(-1)^{\frac{n-1}{2}}2^n  \pi^{\frac{n+1}{2}} \Gamma \left( \frac{n+1}{2} \right).
$$
Moreover, if $s=s_0+\underline{s}\in \mathbb{R}^{n+1}$ is a paravector the term $e^{-is_0 \xi_0}$ extends to the intrinsic entire slice monogenic function $e^{-is \xi_0}$
and we have the Fourier transforms of the kernels $\mathcal{F}^L_n$ and $\mathcal{F}^R_n$:
$$
\mathrm{F}[\mathcal{F}^L_n(s, \cdot)](\xi)
= k_n\  \frac{\bar{\xi}}{\xi_0^2+| \underline{\xi}|^2} \ e^{-is \xi_0},\ \ \ \
\xi_0+\underline{\xi}\not=0
$$
and
$$
\mathrm{F}[\mathcal{F}^R_n(s, \cdot)](\xi)= k_n\ e^{-is \xi_0}\ \frac{\bar{\xi}}{\xi_0^2+| \underline{\xi}|^2},\ \ \ \
\xi_0+\underline{\xi}\not=0.
$$
The extension $\mathrm{F}[\mathcal{F}^L_n(s, \cdot)](\xi)$ is right slice monogenic in $s$, while
$\mathrm{F}[\mathcal{F}^R_n(s, \cdot)](\xi)$ is right slice monogenic in $s$.

\end{theorem}
\begin{proof}
We observe that the Fourier transform of the kernel $\mathcal{F}_n(s, \cdot)$ is meaningful
and from similar computations at the beginning of Theorem \ref{TCR1} we obtain
\[
\begin{split}
\widehat{\mathcal{F}_n}(s,\xi)  &=   \gamma_n s \int_{\mathbb{R}} e^{-ix_0 \xi_0}  (2 \pi)^{\frac{n}{2}} | \underline{\xi} |^{- \frac{n-2}{2}} \! \! \int_{0}^{\infty} (s^2-2x_0s+x_0^2+r^2)^{-\frac{n+1}{2}} J_{\frac{n-2}{2}}(| \underline{\xi}| r) r^{\frac{n}{2}}  dr d x_0
\\
&
 - \gamma_n \int_{\mathbb{R}} x_0 e^{-ix_0 \xi_0} (2 \pi)^{\frac{n}{2}} | \underline{\xi} |^{- \frac{n-2}{2}} \! \! \int_{0}^{\infty} (s^2-2x_0s+x_0^2+r^2)^{-\frac{n+1}{2}} J_{\frac{n-2}{2}}(| \underline{\xi}| r) r^{\frac{n}{2}}  dr d x_0
 \\
&
+ \gamma_n i \sum_{j=1}^n e_j \! \! \int_{\mathbb{R}} e^{-ix_0 \xi_0}   \frac{\partial}{\partial \xi_j} \biggl( (2 \pi)^{\frac{n}{2}} | \underline{\xi} |^{- \frac{n-2}{2}} \! \! \int_{0}^{\infty} (s^2-2x_0s+x_0^2+r^2)^{-\frac{n+1}{2}}  J_{\frac{n-2}{2}}(| \underline{\xi}| r) r^{\frac{n}{2}}  dr d x_0\biggl)
\\
&:= \gamma_n\left(\widehat{\mathcal{F}_{n,1}}(s, \xi)+\widehat{\mathcal{F}_{n,2}}(s, \xi)+\widehat{\mathcal{F}_{n,3}}(s, \xi) \right).
\end{split}
\]
Now, we focus on the first two members
\begin{eqnarray*}
\widehat{\mathcal{F}_{n,1}}(s, \xi)&+&\widehat{\mathcal{F}_{n,2}}(s, \xi)
\\
&=& s \int_{\mathbb{R}} e^{-ix_0 \xi_0}  (2 \pi)^{\frac{n}{2}} | \underline{\xi} |^{- \frac{n-2}{2}} \! \! \int_{0}^{\infty} (s^2-2x_0s+x_0^2+r^2)^{-\frac{n+1}{2}}  J_{\frac{n-2}{2}}(| \underline{\xi}| r) r^{\frac{n}{2}}  dr d x_0
\\
&-& \int_{\mathbb{R}} x_0 e^{-ix_0 \xi_0} (2 \pi)^{\frac{n}{2}} | \underline{\xi} |^{- \frac{n-2}{2}}   \int_{0}^{\infty} (s^2-2x_0s+x_0^2+r^2)^{-\frac{n+1}{2}} J_{\frac{n-2}{2}}(| \underline{\xi}| r) r^{\frac{n}{2}}  dr d x_0
\\
&=& (2 \pi)^{\frac{n}{2}} | \underline{\xi} |^{- \frac{n-2}{2}}  \int_{\mathbb{R}} (s-x_0)e^{-i x_0 \xi_0} \int_{0}^\infty [(s-x_0)^2+r^2]^{- \frac{n+1}{2}} J_{\frac{n-2}{2}}(| \underline{\xi}| r) r^{\frac{n}{2}}  dr d x_0.
\end{eqnarray*}
Firstly, we solve the integral in the variable $r$. From \cite[formula 6.565 (3)]{GR} we know that
\begin{equation}
\label{Inte}
\int_{0}^\infty x^{\nu+1} (x^2+a^2)^{-\nu- \frac{3}{2}}J_{\nu}(bx) dx= \frac{b^\nu \sqrt{\pi}}{2^{\nu +1} |a| e^{|a| b} \Gamma \left( \nu + \frac{3}{2} \right)} \qquad b>0, \quad \nu >-1.
\end{equation}
In our case $b:= | \underline{\xi}|$, $ \nu= \frac{n}{2}-1$ and $a=s-x_0$. Then
\begin{eqnarray*}
\widehat{\mathcal{F}_{n,1}}(s, \xi)+\widehat{\mathcal{F}_{n,2}}(s, \xi) &=& \frac{(2 \pi)^{\frac{n}{2}}2^{- \frac{n}{2}} \sqrt{\pi} | \underline{\xi} |^{- \frac{n}{2}+1}| \underline{\xi} |^{\frac{n}{2}-1}}{\Gamma \left( \frac{n+1}{2} \right)} \int_{\mathbb{R}} e^{-ix_0 \xi_0} \frac{(s-x_0)}{|s-x_0|} e^{-| \underline{\xi} ||s-x_0|} dx_0\\
&=&\frac{\pi^{\frac{n+1}{2}}}{\Gamma \left( \frac{n+1}{2} \right)} \int_{\mathbb{R}} e^{-ix_0 \xi_0} \frac{(s-x_0)}{|s-x_0|} e^{-| \underline{\xi} ||s-x_0|} d x_0.
\end{eqnarray*}
Now, we put $s+y=x_0$. Thus we have
$$ \widehat{\mathcal{F}_{n,1}}(s, \xi)+\widehat{\mathcal{F}_{n,2}}(s, \xi) =-\frac{\pi^{\frac{n+1}{2}}}{\Gamma \left( \frac{n+1}{2} \right)} e^{-is \xi_0}\int_{\mathbb{R}} e^{-iy \xi_0} \frac{y}{|y|} e^{-| \underline{\xi} ||y|} dy.$$
From \cite[formula 3.2 pag 11]{O} we know that
$$ \int_{0}^\infty \cos(x u) \frac{e^{-ax}}{x} dx=- \frac{\log(a^2+u^2)}{2}.$$
Thus by the Euler's formula we get
\begin{equation}
\label{inte2}
\mathrm{F} \left(\frac{1}{|y|} e^{-| \underline{\xi}| |y|} \right) (\xi_0)= 2 \int_{0}^\infty \cos(y \xi_0)\frac{1}{y}e^{-| \underline{\xi}|y} d y=-\log(\xi_0^2+| \underline{\xi}|^2).
\end{equation}
Using basic properties of the Fourier transform we obtain
\begin{eqnarray*}
\int_{\mathbb{R}} e^{-iy \xi_0} \frac{y}{|y|} e^{-| \underline{\xi} ||y|} dy &=&
 \mathrm{F} \left(\frac{y}{|y|} e^{-| \underline{\xi}| |y|} \right) (\xi_0)=i \frac{d}{d \xi_0} \mathrm{F}\left(\frac{1}{|y|} e^{-| \underline{\xi} ||y|} \right) (\xi_0)\\
&=& -i \frac{d}{d \xi_0} \left( \log(\xi_0^2+| \underline{\xi}|^2) \right)=- \frac{2i \xi_0}{\xi_0^2+| \underline{\xi}|^2}.
\end{eqnarray*}
Therefore
$$  \widehat{\mathcal{F}_{n,1}}(s, \xi)+\widehat{\mathcal{F}_{n,2}}(s, \xi)= \frac{2i \pi^{\frac{n+1}{2}} e^{-is \xi_0} \xi_0}{\Gamma\left( \frac{n+1}{2} \right) ( \xi_0^2+| \underline{\xi}|^2)}.$$
Finally, we multiply by $ \gamma_n:= (-1)^{\frac{n-1}{2}} 2^{n-1} \left[\Gamma\left( \frac{n+1}{2} \right) \right]^2$
\begin{equation}
\label{CR4}
\gamma_n (\widehat{\mathcal{F}_{n,1}}(s, \xi)+\widehat{\mathcal{F}_{n,2}}(s, \xi))=\frac{i(-1)^{\frac{n-1}{2}} 2^n  \pi^{\frac{n+1}{2}}  \Gamma\left( \frac{n+1}{2} \right)}{ \xi_0^2+| \underline{\xi}|^2}\xi_0e^{-is \xi_0}.
\end{equation}
Now we compute $ \widehat{\mathcal{F}_{n,3}}(s, \xi)$.
\begin{eqnarray*}
\widehat{\mathcal{F}_{n,3}}(s, \xi) =  i (2 \pi)^{\frac{n}{2}} \sum_{j=1}^n e_j     \frac{\partial}{\partial \xi_j} \biggl(  | \underline{\xi} |^{- \frac{n-2}{2}} \int_{\mathbb{R}} e^{-ix_0 \xi_0} \int_{0}^{\infty} [(s-x_0)^2+r^2]^{-\frac{n+1}{2}}  J_{\frac{n-2}{2}}(| \underline{\xi}| r) r^{\frac{n}{2}}  dr d x_0\biggl).
\end{eqnarray*}
As before we compute the integral in the variable $r$ using \eqref{Inte} with $b:= | \underline{\xi}|$, $ \nu= \frac{n}{2}-1$ and $a=s-x_0$. Thus we have
\begin{eqnarray*}
\widehat{\mathcal{F}_{n,3}}(s, \xi)&=& \frac{i(2 \pi)^{\frac{n}{2}}2^{- \frac{n}{2}} \sqrt{\pi}}{\Gamma \left( \frac{n+1}{2} \right)} \sum_{j=1}^n e_j \frac{\partial}{\partial \xi_j} \left( | \underline{\xi}|^{- \frac{n}{2}+1}| \underline{\xi}|^{ \frac{n}{2}-1} \int_{\mathbb{R}}e^{-ix_0 \xi_0}|s-x_0|^{-1} e^{-| \underline{\xi}||s-x_0|}d x_0\right)  \\
&=& \frac{i \pi^{\frac{n+1}{2}}}{\Gamma \left( \frac{n+1}{2} \right)} \sum_{j=1}^n e_j \frac{\partial}{\partial \xi_j} \left( \int_{\mathbb{R}}e^{-ix_0 \xi_0}|s-x_0|^{-1} e^{-| \underline{\xi}||s-x_0|}d x_0\right)  .
\end{eqnarray*}
We put $s+y=x_0.$
$$ \widehat{\mathcal{F}_{n,3}}(s, \xi) =\frac{i \pi^{\frac{n+1}{2}}e^{-is \xi_0}}{\Gamma \left( \frac{n+1}{2} \right)} \sum_{j=1}^n e_j \frac{\partial}{\partial \xi_j} \left( \int_{\mathbb{R}}e^{-i y \xi_0}|y|^{-1} e^{-| \underline{\xi}||y|} d y\right)  .$$
From \eqref{inte2} we know that
$$ \int_{\mathbb{R}}e^{-i y \xi_0}|y|^{-1} e^{-| \underline{\xi}||y|} dy=-\log(\xi_0^2+| \underline{\xi}|^2). $$
Therefore
\begin{eqnarray*}
\widehat{\mathcal{F}_{n,3}}(s, \xi) &=&-\frac{i \pi^{\frac{n+1}{2}}e^{-is \xi_0}}{\Gamma \left( \frac{n+1}{2} \right)} \sum_{j=1}^n e_j \frac{\partial}{\partial\xi_j} \left( \log(| \underline{\xi}|^2+\xi_0^2) \right)
\\
&=&
- \frac{2i \pi^{\frac{n+1}{2}}e^{-is \xi_0}}{\Gamma \left( \frac{n+1}{2} \right)(\xi_0^2+| \underline{\xi}|^2)} \sum_{j=1}^n e_j \xi_j\\
&=& - \frac{2i \pi^{\frac{n+1}{2}}}{\Gamma \left( \frac{n+1}{2} \right)(\xi_0^2+| \underline{\xi}|^2)}  \, \, \underline{\xi}e^{-is \xi_0}.
\end{eqnarray*}
Finally, multiplying by $\gamma_n:= (-1)^{\frac{n-1}{2}} 2^{n-1} \left[\Gamma\left( \frac{n+1}{2} \right) \right]^2$ we have
\begin{equation}
\label{CR5}
\gamma_n \widehat{\mathcal{F}_{n,3}}(s, \xi)=- \frac{i(-1)^{\frac{n-1}{2}} 2^n \pi^{\frac{n+1}{2}}\Gamma \left( \frac{n+1}{2} \right)}{\xi_0^2+| \underline{\xi}|^2}  \, \, \underline{\xi}e^{-is \xi_0}.
\end{equation}
Putting together \eqref{CR4} and \eqref{CR5} we get
$$\widehat{\mathcal{F}_n}(s,\xi)=k_n \frac{\bar{\xi}}{\xi^2+|\underline{\xi}|^2} e^{-is \xi_0},$$
where $k_n:=i(-1)^{\frac{n-1}{2}}2^n  \pi^{\frac{n+1}{2}} \Gamma \left( \frac{n+1}{2} \right)$.
Finally, the slice monogenic extensions
are obtained reasoning as in the case of the Cauchy kernel.
\end{proof}

\section{The relation of the kernels $S^{-1}$ and $F_n$ via the Fourier transform}
\color{black}{Before to prove a fundamental result we recall that when $n$ is even the operator $\Delta^{\frac{n-1}{2}}$ is defined by the Fourier multipliers
\begin{eqnarray}
\label{fract}
\Delta^{\frac{n-1}{2}} f(x)= \mathcal{R}[(i | \xi|)^{n-1} \mathrm{F}(f(x))(\xi)](x),
\end{eqnarray}
where $ \mathrm{F}$ and $ \mathcal{R}$, are respectively, the Fourier and the inverse Fourier transformations, given, respectively in Definition \ref{Fourier} and Definition \ref{Fourier2}.}

\begin{theorem}
\label{TCR3}
For $x \in \mathbb{R}^{n+1}$ and $s\in \mathbb{R}$ we have that
\begin{equation}
\label{CR6}
\Delta^{\frac{n-1}{2}} S^{-1}(s,x)= \gamma_n (s- \bar{x})(s^2-2x_0s+ |x|^2)^{- \frac{n+1}{2}}
\end{equation}
\end{theorem}
\begin{proof}
If $n$ is odd, this can be proved through pointwise differential computation, see \cite[Thm. 3.3]{CoSaSo}. While for the case $n$ even the result will be showed for any $ \varphi \in \mathcal{S}(\mathbb{R}^{n+1})$.
Firstly we prove the equality for $s \in \mathbb{R}$. The formula \eqref{new3} and Theorem \ref{Plan} imply that we can pass the factional Laplacian to the test function, so we have
$$ \int_{\mathbb{R}^{n+1}} \Delta^{\frac{n-1}{2}} S^{-1}(s, x) \overline{ \varphi(x)} dx = \int_{\mathbb{R}^{n+1}} S^{-1}(s, x) \overline{ \Delta^{\frac{n-1}{2}} \varphi(x)} dx.$$
Using another time Theorem \ref{Plan} we get
$$
\int_{\mathbb{R}^{n+1}} \Delta^{\frac{n-1}{2}} S^{-1}(s, x) \overline{ \varphi(x)} dx=  \frac{1}{(2 \pi)^{n+1}} \int_{\mathbb{R}^{n+1}} \mathrm{F} \left(S^{-1}(s, .) \right) (\xi) \overline{\mathrm{F} \left(\Delta^{\frac{n-1}{2}} \varphi(x) \right)(\xi)} d \xi
$$
From Theorem \ref{TCR1} and Theorem \ref{TCR2} we obtain
\[
\begin{split}
\int_{\mathbb{R}^{n+1}} &\Delta^{\frac{n-1}{2}} S^{-1}(s, x) \overline{ \varphi(x)} dx
\\
&= \frac{1}{(2 \pi)^{n+1}}  i(-1)^{\frac{n-1}{2}} 2^n \pi^{\frac{n+1}{2}} \Gamma \left( \frac{n+1}{2} \right) \int_{\mathbb{R}^{n+1}} \frac{\bar{\xi} e^{-is \xi_0}}{(\xi_0^2+| \underline{\xi}|^2)^{\frac{n+1}{2}}} | \xi |^{n-1} \overline{\hat{\varphi}(\xi)} d \xi
\\
&= \frac{1}{(2 \pi)^{n+1}}  i(-1)^{\frac{n-1}{2}} 2^n \pi^{\frac{n+1}{2}} \Gamma \left( \frac{n+1}{2} \right) \int_{\mathbb{R}^{n+1}} \frac{\bar{\xi} e^{-is \xi_0}}{\xi_0^2+| \underline{\xi}|^2} \overline{\hat{\varphi}(\xi)} d \xi\\
&= \frac{1}{(2 \pi)^{n+1}} \int_{\mathbb{R}^{n+1}} \widehat{\mathcal{F}_n}(s, \xi) \overline{\hat{\varphi}(\xi)} d \xi.
\end{split}
\]
Finally by applying another time the Theorem \ref{Plan} we get
$$ \int_{\mathbb{R}^{n+1}} \Delta^{\frac{n-1}{2}} S^{-1}(s, x) \overline{ \varphi(x)} dx=\int_{\mathbb{R}^{n+1}} \mathcal{F}_n(s, x) \overline{\varphi(x)} dx.$$

\end{proof}
\begin{corollary}
The relation \eqref{CR6} extends to $s \in \mathbb{R}^{n+1}$, considering the left and the right slice monogenic extensions.
\end{corollary}
\begin{proof}
The extension of the equation \eqref{CR6} to $s \in \mathbb{R}^{n+1}$ follows from the Identity principle, because the function $e^{-is \xi_0}$ is trivially intrinsic slice monogenic.
\end{proof}

We conclude with some remarks.

\begin{remark}
Both classes of hyperholomorphic functions have a Cauchy formula that can be used to define
functions of quaternionic operators or of $n$-tuples of
operators that do not necessarily commute.
\end{remark}

\begin{remark}
The Cauchy formula of slice hyperholomorphic functions generates the $S$-functional calculus for quaternionic linear operators
or for $n$-tuples of not necessarily commuting operators, this calculus is based on the notion of
$S$-spectrum. The spectral theorem for quaternionic operators is also based on the $S$-spectrum.
The $S$-spectrum is used in quaternionic and in Clifford operators theory, see \cite{SCHURBOOK,FRACTBOOK,6CKG,NONCOMMBOOK}.
\end{remark}

\begin{remark}
The Cauchy formula of monogenic functions generates the monogenic functional calculus that is based on the monogenic spectrum.
For monogenic operator theory and related topics see \cite{jefferies,6MQ,mitreabook,booktao,6qian1} and the references therein.
The $F$-functional calculus is a bridge between the spectral theory on the $S$-spectrum and the monogenic spectral theory and it is studied in \cite{FCAL1,FCAL2,CoSaSo}.
\end{remark}


\begin{thebibliography}{211}


\bibitem{SCHURBOOK} D. {Alpay}, F. {Colombo},  I. {Sabadini},
 {\em Slice Hyperholomorphic Schur Analysis},
 Operator Theory: Advances and Applications, 256. Birkh\"auser/Springer, Cham, 2016. xii+362 pp.

\bibitem{bds} F. Brackx, R. Delanghe, F. Sommen, {\em Clifford Analysis}, Pitman Res. Notes in Math., 76, 1982.


\bibitem{FCAL1}
 F. Colombo, J. Gantner, {\em Formulations of the F-functional calculus and some consequences}, Proc. Roy. Soc. Edinburgh Sect. A, {\bf 146} (2016),  509--545.

\bibitem{FRACTBOOK}
F. Colombo, J. Gantner,
{\em Quaternionic closed operators, fractional powers and fractional diffusion processes},
 Operator Theory: Advances and Applications, 274. Birkh\"auser/Springer, Cham, 2019. viii+322 pp.



\bibitem{6CKG}
F. Colombo, J. Gantner, D. P. Kimsey,
{\em Spectral theory on the S-spectrum for quaternionic operators},
 Operator Theory: Advances and Applications, 270. Birkh\"auser/Springer, Cham, 2018. ix+356 pp.

\bibitem{radon}
 F. Colombo, R. Lavicka, I. Sabadini,V. Soucek,
 {\em The Radon transform between monogenic and generalized slice monogenic functions}, Math. Ann., {\bf 363} (2015),  733--752.




\bibitem{GLOBAL}
F. Colombo, J.O. Gonzalez-Cervantes, I. Sabadini,
{\em  A nonconstant coefficients differential operator associated to slice monogenic functions}, Trans. Amer. Math. Soc., {\bf 365} (2013), 303--318.


\bibitem{SOREN1}
   F. Colombo, R. S. Krausshar, I. Sabadini,
     {\em Symmetries of slice monogenic functions},
      J. Noncommut. Geom., {\bf 14} (2020), 1075--1106.


\bibitem{A}
 F. Colombo,  D. Pena Pena, I. Sabadini, F. Sommen,
{\em A new integral formula for the inverse Fueter mapping theorem},
  J. Math. Anal. Appl., { \bf 417} (2014),  112--122.

\bibitem{cstrends} F. Colombo, I. Sabadini, {\em A structure formula for  slice monogenic functions and some of its consequences}, Hypercomplex Analysis, Trends in Mathematics, Birkh\"auser, 2009, 69--99.

\bibitem{FCAL2}
 F. Colombo,  I. Sabadini,{\em The F-functional calculus for unbounded operators},
  J. Geom. Phys., {\bf 86} (2014), 392--407.


\bibitem{CoSaSo1}  F. Colombo, I. Sabadini, F. Sommen,
{\em The inverse Fueter mapping theorem},
Commun. Pure Appl. Anal., {\bf 10} (2011), 1165--1181.



\bibitem{CoSaSo2} F. Colombo, I. Sabadini, F. Sommen,
{\em The inverse Fueter mapping theorem in integral form using spherical monogenics},
 Israel J. Math., {\bf 194} (2013),  485--505.

\bibitem{C}
F. Colombo, I. Sabadini, F. Sommen,
{ \em The Fueter primitive of biaxially monogenic functions},
Commun. Pure Appl. Anal., {\bf 13} (2014), 657--672.



\bibitem{BLU}
 F. Colombo, I. Sabadini, F. Sommen, D.C. Struppa,
  {\em Analysis of Dirac Systems and Computational Algebra},
Progress in Mathematical Physics, Vol. 39, {Birkh\"auser}, Boston,
2004.


\bibitem{CoSaSo}  F. Colombo, I. Sabadini, F. Sommen,
{\em The Fueter mapping theorem in integral form and  the $\mathcal{F}$-functional calculus},
 Math. Methods Appl. Sci., {\bf 33} (2010), 2050-2066.

	
\bibitem{Entirebook}
 F. Colombo, I. Sabadini, D.C. Struppa,
{\em Entire Slice Regular Functions},
 SpringerBriefs in Mathematics. Springer, Cham, 2016, v+118.

\bibitem{CSSf}
        F. Colombo,  I. Sabadini, D.C. Struppa,
       {\em Slice monogenic functions},
        Israel J. Math., {\bf 171} (2009), 385--403.


\bibitem{CSSd}
  F. Colombo,  I. Sabadini, D.C. Struppa,
 {\em An extension theorem for slice monogenic functions and some of its consequences},
  Israel J. Math., {\bf 177} (2010), 369--389.

\bibitem{CSSe}
   F. Colombo,  I. Sabadini, D.C. Struppa,
   {\em Duality theorems for slice hyperholomorphic functions},
   J. Reine Angew. Math., {\bf 645} (2010), 85--105.

\bibitem{NONCOMMBOOK} F. Colombo, I. Sabadini, D.C. Struppa, {\em Noncommutative functional calculus. Theory and applications of slice hyperholomorphic functions},
     Progress in Mathematics, 289. Birkh\"auser/Springer Basel AG, Basel, 2011. vi+221 pp.


\bibitem{SCEBOOK} F. Colombo, I. Sabadini, D.C. Struppa,
 {\em Michele Sce's Works in Hypercomplex Analysis.
A Translation with Commentaries}, Cham: Birkh\"auser (ISBN 978-3-030-50215-7/hbk; 978-3-030-50216-4/ebook). vi, 122 p. (2020)

 \bibitem{Cnudde}
L. Cnudde,  H. De Bie,  G. Ren,
{\em Algebraic approach to slice monogenic functions},
Complex Anal. Oper. Theory, {\bf 9} (2015),  1065--1087.


\bibitem{DELSOSOU}
 R. Delanghe, F. Sommen,  V. Soucek,
 {\em Clifford algebra and spinor-valued functions. A function theory for the Dirac operator},
  Related REDUCE software by F. Brackx and D. Constales.
  With 1 IBM--PC floppy disk (3.5 inch). Mathematics and its Applications, 53. Kluwer Academic Publishers Group, Dordrecht, 1992. xviii+485 pp.


\bibitem{CCC}
B. Dong, K. I. Kou, T. Qian, I Sabadini,
{\em The inverse Fueter mapping theorem for axially monogenic functions of degree k},
 J. Math. Anal. Appl., {\bf 476} (2019), 819--835.


\bibitem{D}
B. Dong, K.I. Kou, T. Qian, I. Sabadini,
{\em  On the inversion of Fueter's theorem},
  J. Geom. Phys., {\bf 108} (2016), 102--116.


\bibitem{E1}
D. Eelbode, S. Hohloch, G. Muarem,
{\em The symplectic Fueter-Sce theorem},
 Adv. Appl. Clifford Algebr., {\bf 30} (2020), no. 4, Paper No. 49, 19 pp.

\bibitem{E2}
D. Eelbode, T. Janssens, {\em  Higher spin generalisation of Fueter's theorem},
 Math. Methods Appl. Sci., {\bf 41} (2018), no. 13, 4887--4905.

\bibitem{E3}
D. Eelbode, {\em The biaxial Fueter theorem},
 Israel J. Math., {\bf 201} (2014), 233--245.


\bibitem{fueter1} R. Fueter, {\em Die Funktionentheorie der
Differentialgleichungen $\Delta u = 0$ und $\Delta\Delta u = 0$ mit vier reellen Variablen},  {Comm. Math. Helv.}, {\bf 7} (1934), 307--330.







\bibitem{gimu} J. E. Gilbert, M. A. M. Murray, {\em Clifford algebras and Dirac operators in harmonic analysis,}
Cambridge studies in advanced mathematics n. 26 (1991).



 \bibitem{Gurlebeck:2008}
K. G\"urlebeck, K. Habetha, and W. Spr\"o\ss ig,
\newblock {\em Holomorphic functions in the plane and {$n$}-dimensional space}.
\newblock Birkh\"auser Verlag, Basel, 2008.
\newblock Translated from the 2006 German original.

\bibitem{GR} I.S.~Gradshteyn, I.M.~ Ryzhik, \emph{Table of integrals, Series and Products}, Academic Press LTD, Mathematics/ Engineering, Seventh Edition, 2007.

\bibitem{G} L.~Grafakos, \emph{Classical Fourier Analysis}, Second Edition, Graduate Texts in Math, no. 249, Springer, New York, 2008.

 \bibitem{jefferies} B. Jefferies, {\em Spectral properties of noncommuting operators},
Lecture Notes in Mathematics, 1843, Springer-Verlag, Berlin, 2004.


\bibitem{kqs} K. I. Kou, T. Qian, F. Sommen, {\em Generalizations of Fueter's theorem}, Meth. Appl. Anal., {\bf 9}
(2002), 273--290.

\bibitem{6MQ}
C. Li, A. McIntosh, T. Qian,  {\em Clifford algebras, Fourier transforms and singular convolution operators on Lipschitz surfaces}, Rev. Mat. Iberoamericana, {\bf 10} (1994), 665--721.

\bibitem{mitreabook}
   M. Mitrea,{\em  Clifford wavelets, singular integrals, and Hardy spaces},
    Lecture Notes in Mathematics, 1575. Springer-Verlag, Berlin, 1994. xii+116 pp.

    \bibitem{O} F.~Oberhettinger, \emph{Tables of Fourier transform ad Fourier transform of distributions}, Springer-Verlag, Berlin, 1990.

\bibitem{O1} N.~ Ormerod, \emph{A Theorem on Fourier Transforms of Radial Functions}, J. Math. Anal. Appl {\bf 69} (1979),559-562.

\color{black}{
\bibitem{P1}
  D. Pena Pena, I. Sabadini, F. Sommen,
  {\em Fueter's theorem for monogenic functions in biaxial symmetric domains}, Results Math., {\bf 72} (2017), 1747--1758.


\bibitem{P2} D. Pena-Pena, T. Qian, F. Sommen, {\em An alternative proof of
Fueter's theorem},
Complex Var. Elliptic Equ., {\bf 51} (2006),  913--922.

\bibitem{qian} T. Qian, {\em Generalization of Fueter's result to $\rr^{n+1}$}, Rend. Mat. Acc. Lincei, {\bf 8} (1997), 111--117.

   \bibitem{booktao}
   T. Qian, P. Li,
   {\em  Singular integrals and Fourier theory on Lipschitz boundaries},
 Science Press Beijing, Beijing; Springer, Singapore, 2019. xv+306 pp.

 \bibitem{6qian1}    T. Qian, {\em Singular integrals on star-shaped Lipschitz surfaces in the quaternionic space}, Math. Ann., {\bf 310} (1998), 601--630.

\bibitem{REN4}
G. Ren, X. Wang, {\em Growth and distortion theorems for slice monogenic functions},
 Pacific J. Math., {\bf 290} (2017), 169--198.



\bibitem{sce} M. Sce, {\em Osservazioni sulle serie di potenze nei moduli quadratici}, Atti Acc. Lincei Rend. Fisica, {\bf 23} (1957), 220--225.


  \bibitem{sommen1} F. Sommen, {\em On a generalization of Fueter's theorem}, Zeit. Anal. Anwen., {\bf 19} (2000),
899-902.
}




\end{thebibliography}
\end{document}